\documentclass[letterpaper, 10 pt, conference]{ieeeconf}  

\IEEEoverridecommandlockouts                              

\usepackage{graphics} 
\usepackage{epsfig} 
\usepackage{mathptmx} 
\usepackage{times} 
\usepackage{amsmath} 
\usepackage{amssymb}  

\usepackage[english]{babel}
\usepackage[T1]{fontenc}
\usepackage[utf8]{inputenc}

\usepackage{amsfonts,amsmath,bm,hhline}
\usepackage{epsfig}
\usepackage{color}
\usepackage[hang]{caption2}
\usepackage[normalsize,it,hang]{subfigure}
  
\newtheorem{proposition}{Proposition} 

\DeclareGraphicsExtensions{.jpg,.png,.jpg,.jpg}
\newcommand\tailleFFF{0.33} 
\newcommand\tailleFF{0.4} 

\usepackage{algorithm}
\usepackage{algpseudocode}
\usepackage{siunitx}

\title{\LARGE \bf The Euler-Lagrange equation and optimal control: \\ Preliminary results}

\author{C\'edric Join$^{\dag,\S}$, Emmanuel Delaleau$^{\ddag}$ and Michel Fliess$^{\P,\S}$
\thanks{$^{\dag}$CRAN (CNRS, UMR 7039)), Universit\'{e} de Lorraine, BP 239, 54506 Vand{\oe}uvre-l\`{e}s-Nancy, France (e-mail: \texttt{cedric.join@univ-lorraine.fr}).}%
\thanks{$^{\ddag}$ ENI Brest, UMR CNRS 6027, IRDL, 29200 Brest, France (e-mail: delaleau@enib.fr).}
\thanks{$^{\P}$ LIX (CNRS, UMR 7161), \'Ecole polytechnique, 91128 Palaiseau, France (e-mail: \texttt{michel.fliess@polytechnique.edu}).}%
\thanks{$^{\S}$ AL.I.E.N., 7 rue Maurice Barr\`{e}s, 54330 V\'{e}zelise, France (e-mail: \{\texttt{michel.fliess, cedric.join}\}\texttt{@alien-sas.com})}%
}

\begin{document}

\thispagestyle{empty}
\pagestyle{empty}
\maketitle

\begin{abstract}
Algebraically speaking, linear time-invariant (LTI) systems can be considered as modules. In this framework, controllability is translated as the freeness of the system module. 
Optimal control mainly relies on quadratic Lagrangians and the consideration of any basis of the system module leads to an open-loop control strategy via a linear Euler-Lagrange equation. In this approach, the endpoint is easily assignable and time horizon can be chosen to minimize the criterion.
The loop is closed via an intelligent controller derived from model-free control, which exhibits excellent performances concerning model mismatches and disturbances. The extension to nonlinear systems is briefly discussed.

\end{abstract}

\section{Introduction}

Although the Euler-Lagrange equation is the fundamental equation of the calculus of variations (see, e.g., \cite{courant,gelfand}), its r\^{o}le in optimal control, which 
is an extension of the calculus of variations (see, e.g., \cite{bryson,suss,liber}), is non-existent in most publications with only a few exceptions (see, e.g., \cite{sontag,ortega,fliess00,sira1}). Just look at Kalman's path-breaking contribution to linear quadratic regulators \cite{kalman_mex}.

Kalman writes in the introduction of \cite{kalman_mex}: ``The principal contribution of the paper lies in the introduction and exploitation of the concepts of \emph{controllability} and \emph{observability} \cite{kalman_mos},\footnote{Famous talk at the first IFAC Congress (Moscow, 1960).} with the aid of which we give, for the first time, a complete theory of the regulator problem.'' Our contribution shows that the 35-year-old module-theoretic approach to controllability and observability \cite{fliess90} can be used to start building a new approach to optimal control where the Euler-Lagrange equation becomes dominant. 

For controllable linear time-invariant (LTI) systems we can exhibit results that seem to be quite novel when compared to the huge literature on the linear quadratic regulator (LQR) (see, e.g., \cite{sontag}): 
\begin{enumerate}
    \item Choose a basis $z_1, \dots, z_m$, i.e., a \emph{flat} output, of the corresponding free module \cite{fliess90,fliess00}. Any ``usual'' integral criterion, in particular those in the LQR, may be expressed as a function of the components of the flat output and their derivatives up to some finite order. 
    \item The optimal open-loop control, which is derived via the Euler-Lagrange equation, is a solution of a linear differential equation when the Lagrangian is quadratic. 
    \item This equation permits to impose two-point boundary conditions, i.e., to impose appropriate final conditions.
    \item The integral criterion $J = \int_0^T \mathcal{L}dt$, where $\mathcal{L}$ is a Lagrangian, is a function of the time horizon $T$. A minimum of $J$ corresponds to an optimal horizon $T_0$. This optimality, which is adapted from \cite{erlangen}, \\
    $\bullet$ seems to be new in optimal control; \\
    $\bullet$ is of course meaningless with an infinite time horizon; \\
    $\bullet$ would be questionable without the possibility to choose the final conditions.
    \item The closed-loop is inspired by the HEOL setting \cite{heol}, which is an adaptation of model-free control \cite{mfc1,mfc2}:  \\
        $\bullet$ It is rather easy to implement; \\
        $\bullet$ it exhibits an excellent robustness with respect to model mismatches and disturbances.
\end{enumerate}

For a (differentially) flat nonlinear system \cite{flmr_ijc} a quadratic Lagrangian with respect to the flat output might loose the crucial energy interpretation of the criterion. In order to keep it the Lagrangian is no more quadratic and the Euler-Lagrange equation becomes nonlinear. The corresponding two-point boundary-value problem might be numerically solved via a classic \emph{shooting} method (see, e.g., \cite{keller}).

This communication is organized as follows. Section \ref{algebra} shows how to use rather simple algebraic tools, which unfortunately are today unknown to most control-theorists, in order to express the main properties of LTI systems in an intrinsic way, i.e., independently of any state-variable representation. The Euler-Lagrange equation is examined in Section \ref{el}. Sections \ref{loop} and \ref{exp} provide respectively the presentation of the closed-loop and computer illustration via a simplified electric motor. In Section \ref{nl} an academic example permits a brief discussion on a possible nonlinear generalization. See Section \ref{conclu} for some concluding remarks.

\section{A module-theoretic setting for LTI systems: A short review}\label{algebra}

\subsection{Basics on modules}

This Section recalls some basic notions on module. One refers the reader to \cite{bourbaki,lang,shaf} for details.

Let $\mathbb{R}[\frac{d}{dt}]$ be the commutative ring 
of linear differential operators $\sum_{\rm finite} a_\alpha \frac{d^\alpha}{dt^\alpha}$, $a_\alpha \in \mathbb{R}$. It is isomorphic to the univariate polynomials ring with real coefficients, and, therefore, it is a \emph{principal ideal ring}: Any ideal is generated by a single element. Let $M$ be a finitely generated $\mathbb{R}[\frac{d}{dt}]$-module. Denote ${\rm span}_{\mathbb{R}[\frac{d}{dt}]} (S)$ the submodule generated by a subset $S \subseteq M$. Here are some definitions and properties: 
\begin{itemize}
   \item An element $w \in M$, is said to be \emph{torsion}, if, and only if, there exists $\pi \in \mathbb{R}[\frac{d}{dt}]$, $\pi \neq 0$, such that $\pi w = 0$: in other words $w$ satisfies a homogeneous linear differential equation. The set of all torsion elements in $M$ is the torsion submodule. A module is 
   \begin{itemize}
       \item said to be to torsion if all its elements are torsion; 
       \item is torsion if, and only if, ${\rm dim}_\mathbb{R} (M)$, i.e., its dimension as a $\mathbb{R}$-vector space, is finite.
   \end{itemize}
   \item $M$ is said to be \emph{free} if, and only if, there exists a finite \emph{basis} ${\bf z} = \{z_1, \dots, z_m\}$, where $m = {\rm rk}(M)$ is the \emph{rank} of $M$, such that
   \begin{itemize}
       \item the components of $z$ and their derivatives of any order are $\mathbb{R}$-linearly independent;
       \item any element of $M$ may be expressed as a finite $\mathbb{R}$-linear combination of the components of $z$ and of its derivatives of any order.
   \end{itemize} 
   \item The following decomposition holds for $M$:
   \begin{equation}\label{decomp}
    M \simeq M_{\rm tor} \bigoplus \mathfrak{F}
    \end{equation}
   where $M_{\rm tor}$ (resp. $\mathfrak{F} \simeq M/M_{\rm tor}$) is torsion (resp. free). Set ${\rm rk}(M) = {\rm rk}(\frak{F})$. Thus $M$ is torsion if, and only if, its rank is $0$.
   \item Let $M^\prime \subset M$ be a submodule of $M$. Then ${\rm rk}(M^\prime) \leqslant {\rm rk}(M)$ and ${\rm rk}(M/M^\prime) = {\rm rk}(M) - {\rm rk}(M^\prime)$.
\end{itemize}

\subsection{LTI systems}

A {\em linear time-invariant}, or {\em LTI}, system $\Lambda$ is a finitely generated ${\mathbb{R}}[\frac{d}{dt}]$-module. A LTI \emph{dynamics} is a LTI system equipped with a finite subset ${\bf u} = {\{}u_1, \dots, u_m{\}}$ of \emph{control variables}, such that the quotient module $V_{\mathbb{R}} = \Lambda / {\rm span}_{\mathbb{R}[\frac{d}{dt}]} ({\bf u})$ is torsion. The control variables are said to be \emph{independent} if, and only if, the module ${\rm span}_{\mathbb{R}[\frac{d}{dt}]} ({\bf u})$ is free of rank $m$: this property is assumed to hold in the sequel. Then ${\rm rk}(\Lambda) = m$. A LTI input-output system is a LTI dynamics equipped with a finite subset ${\bf y} = {\{}y_1, \dots, y_p{\}}$ of \emph{output variables}.

\subsubsection*{Stability} The tensor product $V_{\mathbb{C}} = {\mathbb{C}} \otimes_\mathbb{R} V_{\mathbb{R}}$, where $\mathbb C$ is the field of complex numbers, is a $\mathbb C$-vector space, which is of the same finite dimension as $V_{\mathbb{R}}$. The derivative $\frac{d}{dt}$ is a ${\mathbb{C}}$-linear endomorphism of $V_{\mathbb{C}}$. Dynamics $\Lambda$ is said to be
\emph{stable} if, and only if, the real parts of the eigenvalues of this endomorphism are strictly negative; \emph{unstable} in the other cases.

\subsubsection*{Input-output invertibility} The input-output system $\Lambda$ is said to be \emph{input-output invertible} if, and only if, the quotient module $W_{\mathbb{R}} = \Lambda / {\rm span}_{\mathbb{R}[\frac{d}{dt}]} ({\bf y})$ is torsion, or $\{0\}$. The invertible system $\Lambda$ is said to be 
\begin{itemize}
    \item \emph{non-minimum phase} if, and only if, the the inverse dynamics with control variables $\bf y$ is unstable or marginally stable; 
    \item \emph{minimum phase} if not.
\end{itemize}

\subsection{Controllability} 
A LTI system $\Lambda$ is said to be {\em controllable} if, and only if, it is a free module. See \cite{fliess90} for the equivalence with Kalman's controllability for the Kalman's state-variable representation.\footnote{For an uncontrollable system the torsion submodule in Eq. \eqref{decomp} corresponds to Kalman's uncontrollable subspace.} Any basis ${\bf z} = \{z_1, \dots, z_m\}$, which may be seen as a fictitious output, is called a \emph{flat}, or \emph{basic}, output. Take a controllable dynamics $\Lambda$ with independent control variables ${\bf u} = \{u_1, \dots, u_m\}$. Then ${\rm rk}(\Lambda) = m$ since $\Lambda / {\rm span}_{\mathbb{R}[\frac{d}{dt}]} ({\bf u})$ is torsion.

\subsubsection*{Controllability canonical form} The well-known controllable canonical form
\begin{equation}\label{ex1}
\begin{cases}
\dot{x}_1 &= x_2 \\
&\vdots \\
\dot{x}_n &= -a_0 x_1 - \cdots - a_{n-1}x_n + b u  
\end{cases}
\end{equation}
where $a_0, \dots, a_{n-1}, b \in \mathbb{R}$, $b \neq 0$, corresponds to a
controllable dynamics $\Lambda_1$ with a single control variable $u$. Thus ${\rm rk} (\Lambda_1) = 1$, $x_1$ is a basis of $\Lambda_1$, $n = \dim_{\mathbb{R}} (\Lambda_1 / {\rm span}_{\mathbb{R}[\frac{d}{dt}]} (u))$. 
\subsubsection*{Remark}
Brunovsk\'y's canonical form \cite{brunov}, which is a multivariable extension of Eq.~\eqref{ex1}, may also be obtained via module-theoretic considerations \cite{fliess93}.

\subsection{Observability}
Consider an input-output system $\Lambda$ with control (resp. output) variables $\bf u$ (resp. $\bf y$). It is said to be \emph{observable} if, and only if, $\Lambda = {\rm span}_{\mathbb{R}[\frac{d}{dt}]} ({\bf u}, {\bf y})$: Any system variable may expressed as a $\mathbb{R}$-linear combination of the control and output variables and their derivatives up to some finite order.
See \cite{fliess90} for the equivalence with Kalman's observability for the Kalman's state-variable representation. 
\subsubsection*{Observability with respect to a flat output} Consider a controllable input-output system $\Lambda$. Assume that its output $\bf y$ is flat. Then $\Lambda = {\rm span}_{\mathbb{R}[\frac{d}{dt}]} ({\bf y})$: $\Lambda$ is also observable.

\subsection{Transfer matrix\protect\footnote{See \cite{fliess94} for more details.}} Take a system with $m$ independent control variables $\bf u$ and $p$ output variables $\bf y$. Let ${\mathbb R} (\frac{d}{dt})$ be the quotient field of the ${\mathbb R} [\frac{d}{dt}]$: it is the field of rational functions with real functions in the indeterminate $s = \frac{d}{dt}$. The tensor product ${\mathbb R}(s) \otimes_{\mathbb R} \Lambda$ is a ${\mathbb R}(s)$-vector space. Its dimension is equal to ${\rm rk} (\Lambda) = m$. Set $\hat{u}_\alpha = 1 \otimes_{\mathbb R} u_\alpha$, $\alpha = 1, \dots, m$, $\hat{y}_\beta = 1 \otimes_{\mathbb R} y_\beta$, $\beta = 1, \dots, p$. It yields 
\begin{equation*}
    \begin{pmatrix}
        \hat{y}_1 \\
        \vdots \\
        \hat{y}_p
    \end{pmatrix} = \mathcal{M}
    \begin{pmatrix}
        \hat{u}_1 \\
        \vdots \\
        \hat{u}_m
    \end{pmatrix}
\end{equation*}
where ${\mathcal{M}} \in {\mathbb{R} (s)^{p \times m}}$ is the {\em tranfer matrix}, since $\hat{u}_1, \dots, \hat{u}_m$ is a basis of ${\mathbb R}(s) \otimes_{\mathbb R} \Lambda$.

\subsubsection*{Transfer function} If $m = p = 1$, the tranfer matrix $\mathcal M$ reads $\frac{a}{b}$, where $a, b \in {\mathbb R}[s]$, $b \neq 0$, are coprime polynomials. The corresponding input-output system is minimum phase if, and only if, one of the two following properties {is satisfied}:
\begin{itemize}
    \item $a$ is constant: $y_1$ is a flat output; 
    \item the real parts of the roots of $a$  are strictly negative.
\end{itemize}

\subsubsection*{Remark} In the control literature, transfer matrices and functions are defined via a rather involved analytic tool, the Laplace transform (see, e.g., \cite{sontag}).

\section{Linear Euler-Lagrange equations}\label{el}
\subsection{Peculiar Lagrangians}
Let ${\bf z} = \{z_1, \dots, z_m\}$ be a flat output, i.e., a basis of a controllable linear system $\Lambda$. Take a {\em Lagrangian}, or \emph{cost function}, $\mathcal{L}(z_, \dot{z}_1, \dots, z_{1}^{(\nu_1)}, \dots, z_m, \dot{z}_m, \dots, z_{m}^{(\nu_m)})$. It is assumed to be
\begin{itemize}
    \item a positive-definite quadratic form (see, e.g., \cite{courant,lang,matrices} with respect to the variables $\{z_{\lambda}^{(\mu_\lambda)} - \chi_{\lambda}^{\mu_\lambda}\}$, where $\chi_{\lambda}^{\mu_\lambda} \in \mathbb{R}$ is a constant; 
     \item time-invariant, i.e., with constant coefficients in $\mathbb R$.
\end{itemize}
To find a minimum of the criterion
\begin{equation}
\label{cost}    
J = \int_{0}^{T} \mathcal{L}(z_1 (t), \dots, z_{1}^{(\nu_1)} (t), \dots, z_m (t), \dots, z_{m}^{(\nu_m)} (t)) dt
\end{equation}
where $T > 0$ is the \emph{time horizon}, is associated the Euler-Lagrange system of ordinary differential equations
\begin{equation}
\label{ele}
\frac{\partial \mathcal{L}}{\partial z_\iota} - \frac{d}{dt} \frac{\partial \mathcal{L}}{\partial \dot{z}_\iota} + \dots + (- 1)^{\nu_\iota} \frac{d^{\nu_\iota}}{dt^{\nu_\iota}} \frac{\partial \mathcal{L}}{\partial z_{\iota}^{(\nu_\iota)}} = 0, \hspace{0.16cm} \iota = 1, \dots, m
\end{equation}
Quite straightforward computations yield
\begin{proposition}
 If the Lagrangian $\mathcal{L}$ is a time-invariant positive-definite quadratic form in the variables $(z_{\lambda}^{(\mu_\lambda)} - \chi_{\lambda}^{\mu_\lambda})$, $\chi_{\lambda}^{\mu_\lambda} \in \mathbb{R}$, the Euler-Lagrange equation may be written 
 \begin{equation*}
 \begin{pmatrix}
     H_1 \\
    \vdots \\
    H_m \\
 \end{pmatrix} = 2 \begin{pmatrix}
     \chi^{0}_1 \\
    \vdots \\
    \chi^{0}_m \\
 \end{pmatrix}
\end{equation*}
 where $H_1, \dots, H_m \in {\rm span}_{{\mathbb{R}}[\frac{d}{dt}]}({\bf z})$. The above system of linear differential equations is thus homogeneous if, and only if,  $\chi^{0}_1 = \dots, \chi^{0}_m = 0$. 
\end{proposition}
\subsection{A connection with classic LQR} 
Consider the usual state-variable representation of a linear dynamics
\begin{equation}\label{dyna}
\frac{d}{dt} \begin{pmatrix}
x_1 \\
\vdots \\
x_n \\
\end{pmatrix}  = A  \begin{pmatrix}
x_1 \\
\vdots \\
x_n \\
\end{pmatrix} + B \begin{pmatrix}
u_1 \\
\vdots \\
u_m \\
\end{pmatrix}
\end{equation}
where \begin{itemize}
    \item ${\bf x} = (x_1, \dots, x_n)$ are the components of the state, 
    \item $A \in \mathbb{R}^{n \times n}$, $B \in \mathbb{R}^{n \times m}$ are constant matrices.
\end{itemize}

Introduce the criterion
\begin{equation}\label{lqr}
J_{\rm LQR} = \int_{0}^{T} \left( (x_1, \dots, x_n) Q \begin{pmatrix}
x_1 \\
\vdots \\
x_n \\
\end{pmatrix} + (u_1, \dots, u_m) R \begin{pmatrix}
u_1 \\
\vdots \\
u_m \\
\end{pmatrix} \right) dt
\end{equation}
where $T > 0$, $Q \in \mathbb{R}^{n \times n}$, $R \in \mathbb{R}^{m \times m}$ are constant matrices. The next proposition provides a connection with classic linear quadratic regulators (LQR):
\begin{proposition}\label{prop2}
Assume that
\begin{itemize}
    \item Eq. \eqref{dyna} defines a controllable dynamics,
    \item $Q$ and $R$ are positive-definite (see, e.g., \cite{matrices}).
\end{itemize}
Then the cost function in Eq. \eqref{lqr} may be expressed as a positive-definite quadratic form in the components of any flat output and their derivatives up to some finite order.
\end{proposition}
\begin{proof}
\end{proof}

\subsection{Two-point boundary-value problem}\label{tpbv}
Associate to Eq. \eqref{ex1} the criterion
\begin{equation}\label{J1}
    J_1 = \int_{0}^{T} \mathfrak{q}(x_1, \dots, x_n) + {\mathfrak{r}} u^2
\end{equation}
where $\mathfrak q$ is a positive-definite quadratic form with respect to the state variables, $\mathfrak r$ a real constant. Formula \eqref{ele} shows that the order with respect to $z = x_1$ of the corresponding linear homogeneous Euler-Lagrange equation is $2n$. The dimension of the $\mathbb R$-vector space of its solutions is equal to $2n$.  Let  $\{s_1 (t), \dots, s_{2n}(t) \}$ be a basis. Then $x_1 (t) = \sum_{\mu = 1}^{2n} c_\mu s_\mu (t)$, where $c_\mu \in \mathbb{R}$ is constant. The following two-point boundary conditions, i.e., numerical values for $x_\alpha (0)$, $x_\alpha (T)$, $\alpha = 1, \dots, n$ yield the following linear matrix equation for determining $c_\mu$:
\begin{equation}\label{sing}
\begin{pmatrix}
        x_1 (0) \\
        \vdots \\
        x_{n}(0) \\
        x_1 (T) \\
        \vdots \\
        x_n (T)
    \end{pmatrix} =
     \begin{pmatrix}
        s_1(0) & \cdots  & s_{2n}(0) \\
        \dot{s}_1(0) & \cdots & \dot{s}_{2n}(0) \\
         \cdots \\
         \frac{d^{n}}{dt^{n}}s_1(0) & \cdots & \frac{d^{n}}{dt^{n}}s_{2n}(0) \\
          s_1(T) & \cdots  & s_{2n}(T) \\
        \dot{s}_1(T) & \cdots & \dot{s}_{2n}(T) \\
         \cdots \\
         \frac{d^{n}}{dt^{n}}s_1(T) & \cdots & \frac{d^{n}}{dt^{n}}s_{2n}(T) 
    \end{pmatrix}
    \begin{pmatrix}
        c_1 \\
        \vdots \\
        c_{2n}
    \end{pmatrix}
\end{equation}
The above computations which are adapted from \cite{ince} yield the following results:
\begin{proposition}
If the square $2n \times 2n$ matrix in Eq. \eqref{sing} is invertible, i.e., if its determinant is nonzero, the coefficients $c_\mu$, $\mu = $ are given by the Cramer rule. The invertibility is independent of the choice of the basis $s_\mu (t)$.
\end{proposition}
 
\subsubsection*{Application to time horizon optimality}
Consider, in Eq. \eqref{J1}, $J_1$ as a function of $T$. A minimum of $J_1$ for $T = T_0$ should be interpreted as an {\em optimal time horizon}. 

More general criteria and the multivariable case will be analyzed elsewhere.

\section{Closing the loop}\label{loop}
\subsection{Homeostat}
Associate to a LTI system $\Lambda$ the \emph{variational} system $\Lambda_\Delta$ via the $\mathbb R[\frac{d}{dt}]$-module variational isomorphism $\Delta: \Lambda \rightarrow \Lambda_\Delta$. For any system variable $w \in \Lambda$, $\Delta w^{(\nu)} = \frac{d^\nu}{dt^\nu} \Delta w$. For simplicity, we restrict ourselves to the monovariable case: 
\begin{itemize}
    \item System $\Lambda$, with a control (resp. output) variable $u$ (resp. $y$) is controllable and observable. 
         \item $\Lambda$ is minimum phase.\footnote{Unavoidable numerical approximations might lead to instabilities with a non-minimum phase system (see \cite{mfc1} for details).}
\end{itemize}
From $\sum_{\rm finite} a_\alpha y^{(\alpha)} = \sum_{\rm finite} b_\beta u^{(\beta)}$, $a_\alpha, b_\beta \in \mathbb{R}$, the variational system reads:
\begin{equation}\label{variat}
  \sum_{\rm finite} a_\alpha \frac{d^\alpha}{dt^\alpha} \Delta y = \sum_{\rm finite} b_\beta \frac{d^\beta}{dt^\beta} \Delta u  
\end{equation}
The minimum-phase property yields $b_0 \neq 0$. Select $\nu\geqslant 1$ such that $a_{\nu} \neq 0$. Rewrite Eq. \eqref{variat}
\begin{equation}\label{homeo1}
   \frac{d^{\nu}}{dt^{\nu}} \Delta y = \mathfrak{F} + \mathfrak{a} \Delta u
\end{equation}
where $\mathfrak{a} = \frac{b_0}{a_{\nu}}$, $\mathfrak{F} = \frac{\sum_{\beta \neq 0} b_\beta \frac{d^\beta}{dt^\beta} \Delta u - \sum_{\alpha \neq \nu} a_\alpha \frac{d^\alpha}{dt^\alpha} \Delta y}{a_\nu}$.
When $\mathfrak{F}$ takes also in account model mismatches and disturbances, call Eqn. \eqref{homeo1} \emph{homeostat} \cite{heol} which replaces the ultra-local model in model-free control \cite{mfc1,mfc2}. The numerous successful concrete case-studies in the existing literature suggest that $\nu = 1 \ {\rm or} \ 2$ are the only values which matter.

\subsection{iP and iPD controller}

\subsubsection*{Case $\nu = 1$ with iP controller} The loop is closed by an \emph{intelligent proportional} controller, or \emph{iP},
\begin{equation}
\label{ip}
\Delta u = - \frac{\mathfrak{F}_{\text{est}} + K_P \Delta y }{\mathfrak{a}}    
\end{equation}
where 
\begin{itemize}
    \item the estimate $\mathfrak{F}_{\text{est}}$ is given \cite{mfc1} by
   \begin{footnotesize}
    \begin{equation*}
\label{estim1} 
\begin{aligned}\mathfrak{F}_{\rm est} =& - \frac{6}{\tau^3} \int_{0}^{\tau} \left( (\tau - 2 \sigma)\Delta y(t-\tau+\sigma)\right.\\ & \left.\mbox{}- \sigma (\tau - \sigma)\alpha(t-\tau+\sigma)\Delta u(t-\tau+\sigma)\right)d\sigma
\end{aligned}
\end{equation*}
\end{footnotesize}
\item $\tau > 0$ is ``small,''
\item the gain $K_P \in \mathbb{R}$ ensures local stability if the estimate is ``good'':
$\mathfrak{F} - \mathfrak{F}_{\rm est} \approx 0$.
\end{itemize}

\subsubsection*{Case $\nu = 2$ with iPD controller} The iP \eqref{ip} is replaced by an \emph{intelligent proportional-derivative} controller, or \emph{iPD}, 
\begin{equation}
\label{ipd}
\Delta u = - \frac{\mathfrak{F}_{\text{est}} + K_P \Delta y + K_D \frac{d}{d\sigma} \Delta y}{\mathfrak{a}}    
\end{equation}
where the estimate is given \cite{mfc2} by
\begin{itemize}
    \item    \begin{footnotesize}\begin{eqnarray*}
    \mathfrak{F}_{\rm est} &=& \dfrac{60}{\tau^5}\left[
            \int_0^\tau \left(\big(\tau-\sigma\big)^2 -4\big(\tau-\sigma\big)\sigma + \sigma^2\right)\Delta y(t-\tau+\sigma)
            \right.\nonumber \\
    &&      \mbox{} -
       \left.\dfrac{1}2 \int_0^\tau (\tau-\sigma)^2\sigma^2\alpha (t-\tau+\sigma)\Delta u(t-\tau+\sigma)d\sigma\right]
  \end{eqnarray*}
  \end{footnotesize}
\item  $\tau > 0$ is again ``small,''
the gains $K_P, K_D \in \mathbb{R}$ ensure local stability if the estimate is ``good,''  
\item see Riachy's trick in \cite{mfc2,heol} for avoiding the estimation of the derivative of $\Delta y$.  
\end{itemize}

\section{Computer experiments}\label{exp}
\subsection{General principles}
Take a nominal controllable LTI system $\Lambda$. Section \ref{el} yields via a flat output an optimal nominal trajectory ${\mathfrak{w}}^*$ for any system variable $\mathfrak w$. In order to mitigate model mismatches and disturbances, the corresponding open-loop nominal control ${\bf u}^* = \{u_{1}^*, \dots, u_{m}^*\}$ needs to be completed by an appropriate closed-loop control as described in Sect. \ref{loop}. 

\subsection{A simplified DC motor}\label{motor}

\subsubsection*{Modeling} Follow \cite{fliess00} and consider a linear modeling \cite{babary} of a DC motor\footnote{See \cite{delaleau} for a general nonlinear modeling.} $J_m \dot\omega  = K_ti - \tau_l$, $L_m \frac{di}{dt} = u_m - R_mi - K_b\omega$
where $\omega$ is the angular speed ([\SI{}{\radian\per\second}]),
$\tau_l$ is the disturbance load torque ([\SI{}{\N\m}]), $i$ is the
motor current ([\SI{}{\ampere}]) and $u_m$ is the control voltage
([\SI{}{\V}]).
The parameters are: $J_m$ the rotor inertia
(\SI{2.7}{kg\square\m}), $K_t$ the torque constant
(\SI{2.62}{\N\m\per\radian\second}), $L_m$ the coil inductance
(\SI{153}{\milli\henry}), $R_m$ the coil resistance (\SI{.465}{\ohm})
and $K_b$ the back emf constant
(\SI{2.62}{\V\per\radian\second}).
Set
$a = \frac{K_t}{J_m} = \SI{0.970}{}$,
$b = \frac{K_b}{L_m} = \SI{171}{}$,
$c = \frac{R_m}{L_m} = \SI{30.3}{}$,
$d = \frac{1}{L_m} = \SI{65.1}{}$ and
$e = \frac{1}{J_m} = \SI{0.370}{}$. The DC bus voltage that feeds the chopper is $u_\mathrm{dc} = \SI{500}{V}$ and one can introduce the duty cycle $u$, $0<u<1$ such that $u_m = u_\mathrm{dc} u$. One obtains the state
representation as:
\begin{subequations}\label{dcm}
\begin{eqnarray}
  \dot x_1 &=& ax_2  - e \, \tau_l\\ 
  \dot x_2 &=& -bx_1 - cx_2 + d\,u_\mathrm{dc} \, u\\
  y        &=& x_1
\end{eqnarray}
\end{subequations}
with state variables $x_1 = \omega$ and $x_2 = i$. The nominal model
($T_l\equiv 0$) is controllable and $y$ is a flat output:
\begin{equation}\label{dc}
  u =\frac{1}{d\,u_\mathrm{dc}}\left(\frac{\ddot y + c\dot y}{a}+b y \right)
\end{equation}

\subsubsection*{Criterion} Set
\begin{equation}
  \label{eq:dc:criterion}
  J_1 = \int_0^T \left( u^2 + (y_f - y)^2 + \dot y^2 + y^2\right) dt
\end{equation}
where $y_f$ is the final objective. The initial (resp. final)
conditions are $y(0) = \dot{y}(0) = 0$ (resp. $y(T) = y_f$,
$\dot{y}(T) = 0$). Consider $J_1$ as a function of the
horizon~$T$. Fig. \ref{To} displays a minimum at time
$T_0 \approx \SI{2.3}{\second}$: This is the optimal horizon.

\subsubsection*{iP control} Set $T = \SI{3}{\second}$, a sampling time equal
to $\SI{10}{\milli\second}$, and in \eqref{ip} $K_P=5$,
$\alpha=\frac{ad}{c}u_\mathrm{dc} = \SI{105}{}$.
An additive measurement white
Gaussian noise $\mathcal{N}(0,1)$ is added. Take the disturbance
torque as $\tau_l(t) = 0$, if $t\leqslant \SI{2}{s}$, $\tau_l(t) =300\sin(\pi t)$
if $\SI{2}{s} < t \leqslant \SI{3}{s}$. Assume also a $50\,\%$ uncertainty on the parameter~$d$. Fig.~\ref{CPI} displays excellent performances.

\color{black}

\begin{figure*}[!ht]
\centering
\subfigure[\footnotesize $u^\star$]
{\epsfig{figure=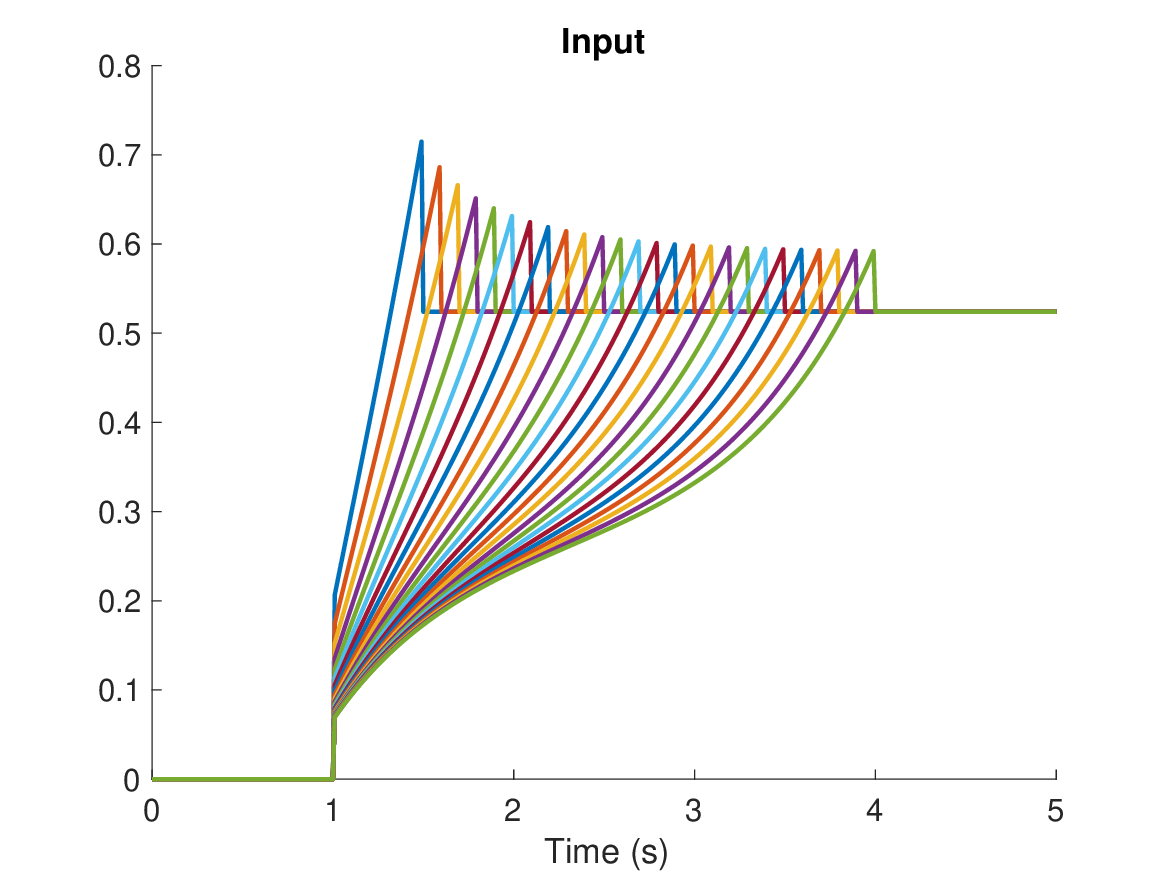,width=\tailleFFF\textwidth}}
\subfigure[\footnotesize $y^\star$]
{\epsfig{figure=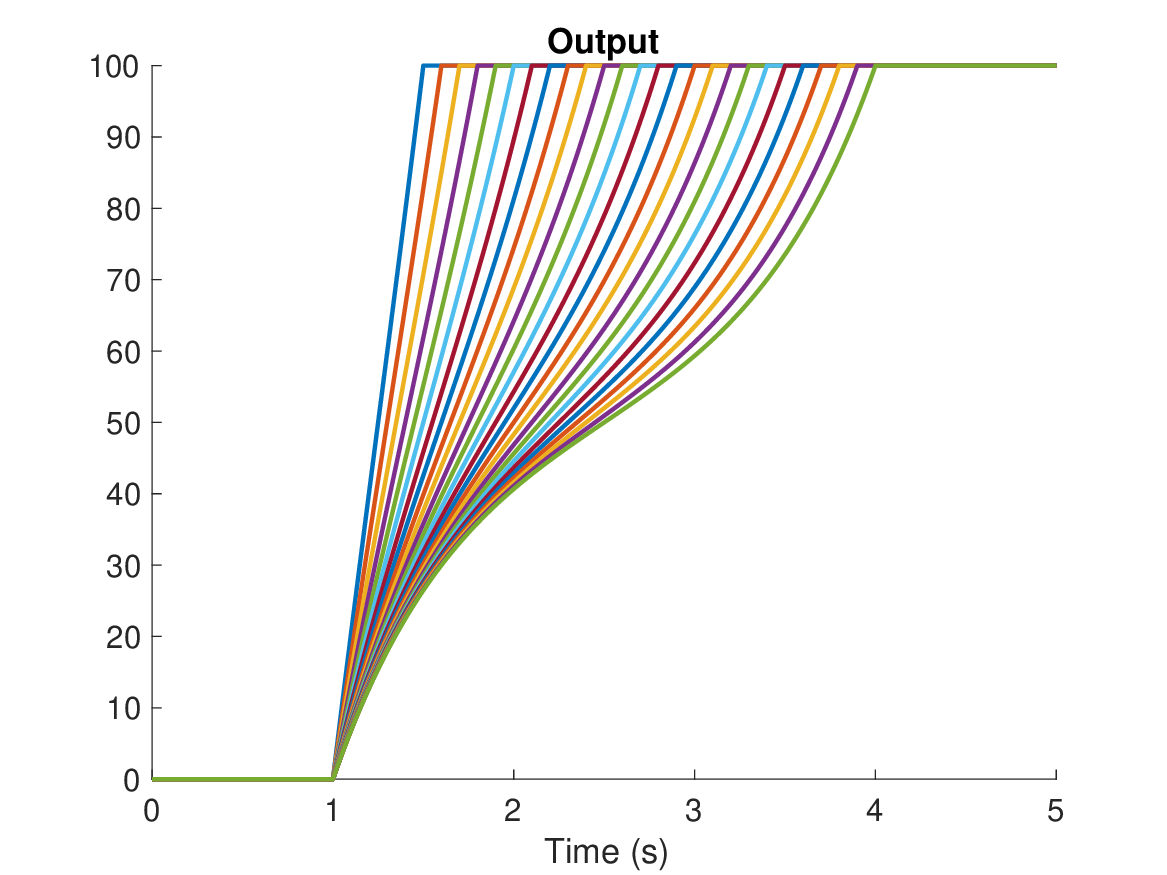,width=\tailleFFF\textwidth}}
%
\subfigure[\footnotesize J]
{\epsfig{figure=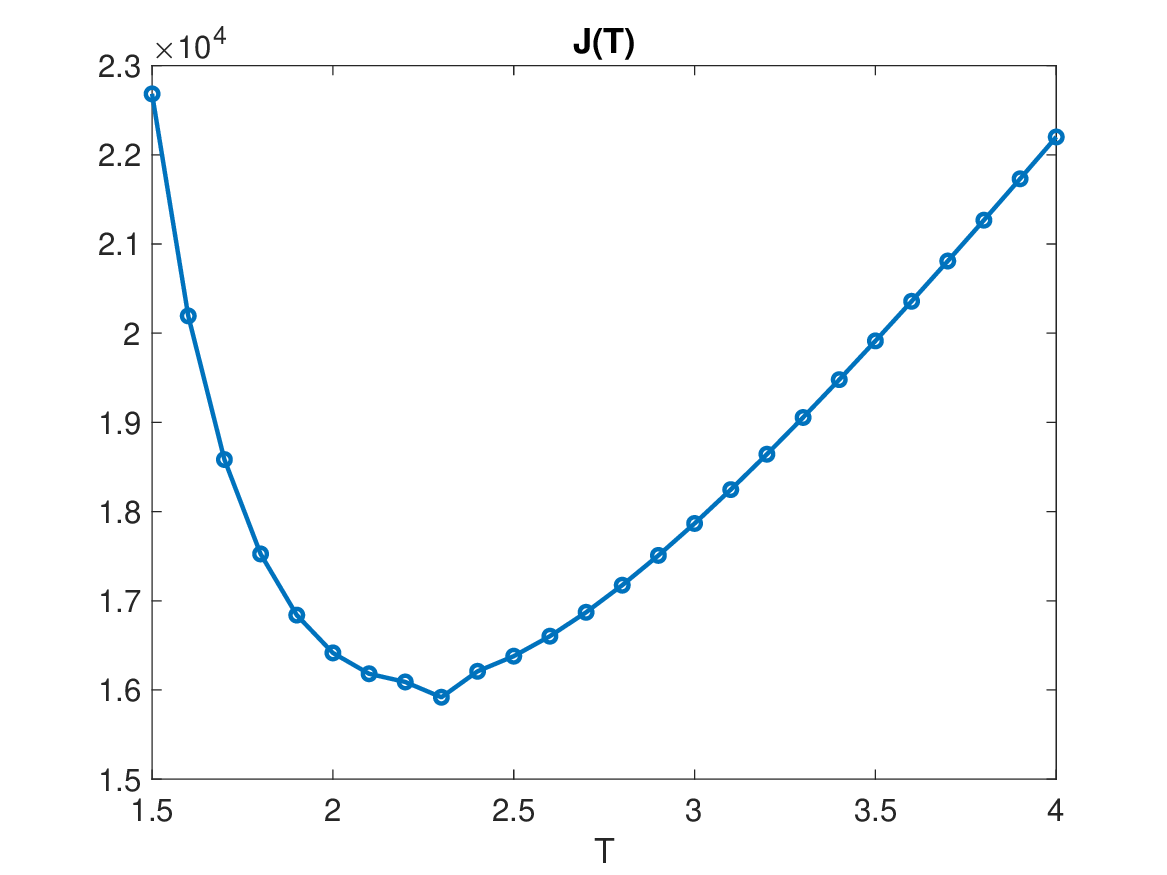,width=\tailleFFF\textwidth}}
\caption{{DC motor optimal time $T$}}\label{To}
\end{figure*}



\begin{figure*}[!ht]
\centering
\subfigure[\footnotesize {$u^\star$ (dotted) --- $u$ (red)}]
{\epsfig{figure=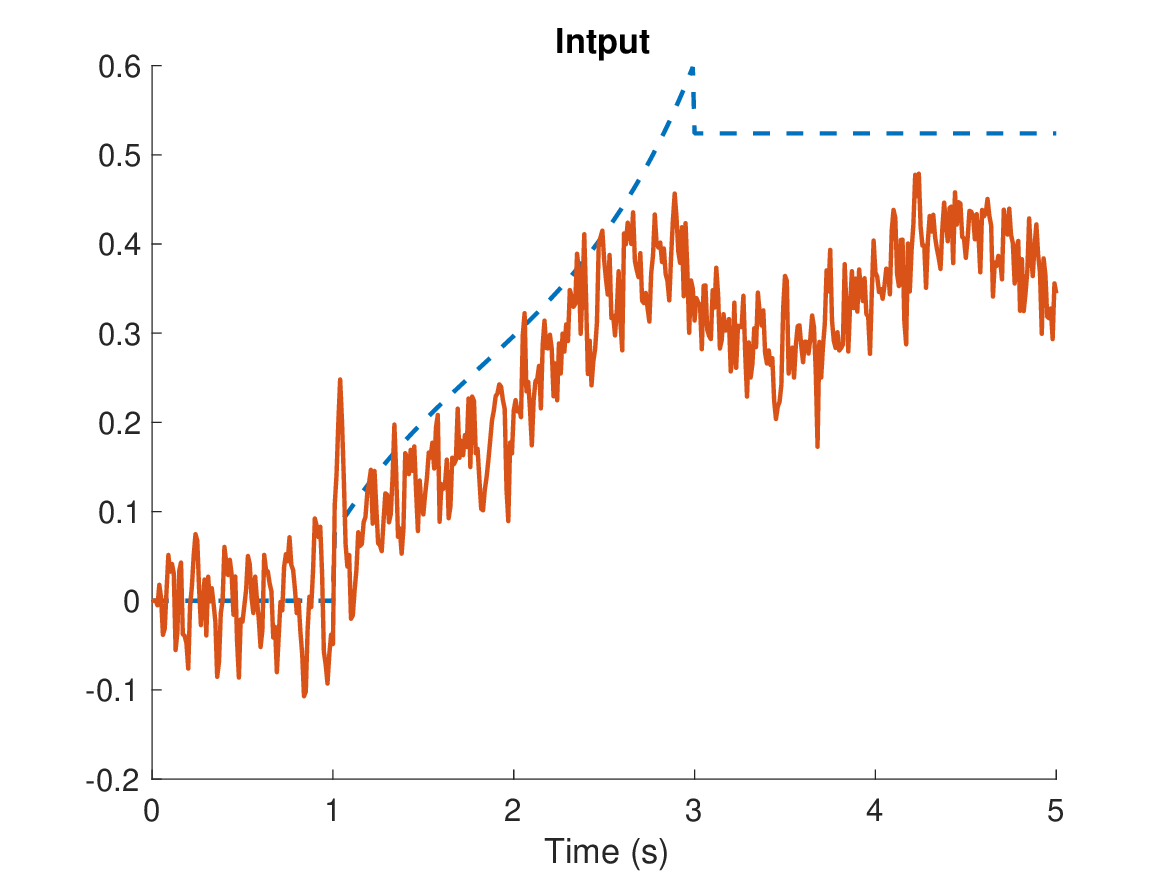,width=\tailleFF\textwidth}}
\subfigure[\footnotesize {$y^\star$ (dotted) --- $y$ (red)}]
{\epsfig{figure=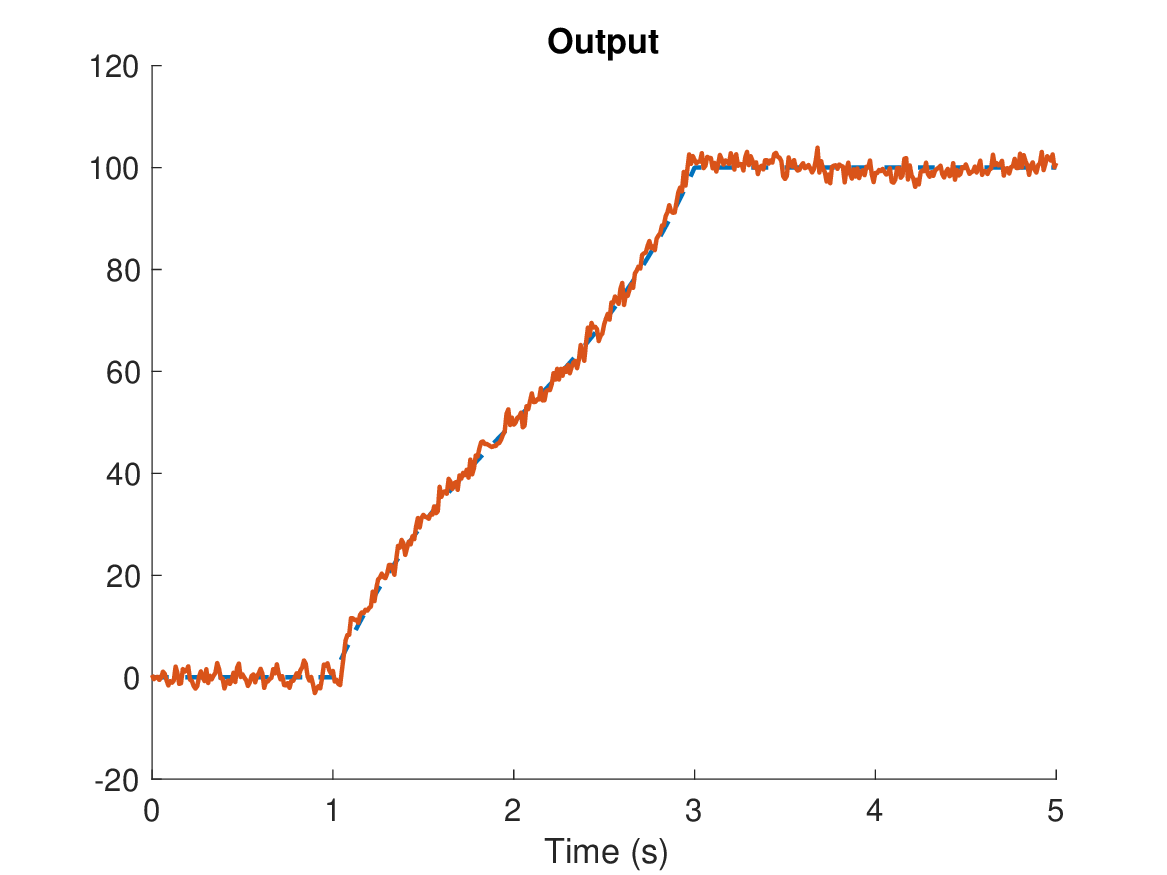,width=\tailleFF\textwidth}}
\caption{{DC motor:} Closed-loop with uncertainty and {disturbance}}\label{CPI}
\end{figure*}

\section{A brief discussion of nonlinear systems}\label{nl}
\subsection{General principles} Let $\Sigma$ be a flat systems: there exists a flat output ${\bf z} = \{z_1, \dots, z_m\}$ such that (see \cite{heol} for details and references)
\begin{itemize}
    \item any system variable may be expressed as a \emph{differential function} of the components of $\bf z$, i.e., as a function of $\bf z$ and the derivatives of its components up to some finite order,
    \item the components of $\bf z$ are differentially independent, i.e., they are not related by any differential relation.
\end{itemize}
The Lagrangian may thus be expressed as a differential function of $\bf z$. The corresponding Euler-Lagrange equation is in general nonlinear. With appropriate boundary conditions it yields an open-loop control strategy. The homeostat \cite{heol}, which is built via the time-varying linearized system around the corresponding trajectory, permits to construct the stabilizing iP or iPD.

\subsection{An elementary academic example}

The system $\mathfrak{T} \dot y + y^3 = \mathfrak{K} u$  where $u$ (resp. $y$) is the control (resp. output) variable, and $\mathfrak{T} = 2$, $\mathfrak{K} = 0.1$ are two constant parameters, 
is obviously flat: $u$ is a differential function of the flat output $y$
\begin{equation}\label{nl2}
 u =\frac{\mathfrak{T} \dot y + y^3}{\mathfrak{K}}   
\end{equation}
\subsubsection*{Comparison between two Lagrangians}
Eq. \eqref{nl2} shows that the Euler-Lagrange equation associated to an energy-oriented Laplacian $\mathcal{L}_{\rm energ} = (y_{f}-y)^2 + {u}^2$, where $y_f = y(T)$, is nonlinear with respect to $y$. It reads 
$$y +\frac{1}{\mathfrak{K}^2}\left( 3 {y}^5-\mathfrak{T}^2\ddot y\right) =  y_f$$
A classic shooting method \cite{keller} permits to solve it at once with a time horizon $T = 3$ and the boundary conditions $y(0) = 0.1$, $y(3) = 1.5$. It yields $\int_0^3 \mathcal{L}_{\rm energ}dt \simeq 1.11.10^3$. 

With a Lagrangian 
$\mathcal{L}_{\rm lin} = (y_{f}-y)^2 + {\dot y}^2$, which is similar to those in Sec. \ref{el} but, here, without a clear physical meaning, we get again a linear Euler-Lagrange equation, which is straightforward to solve analytically. Let us plug therefore the corresponding numerical values of $u$, $y$ in $\mathcal{L}_{\rm energ}$ and integrate again between $0$ and $3$. We obtain $2.05.10^3$. 

There is no question about the superiority of the first approach.

\subsubsection*{Homeostat}
Set $\Delta y= y-y^\star$, $\Delta u=u-u^\star$, $u^\star$, $y^\star$ correspond to the optimal solution with $\mathcal{L}_{\rm energ}$. Introduce the homeostat 
$\frac{d}{dt}\Delta y=F+K\Delta u$.
Instead of the iP \eqref{ip}, introduce 
$$u=\frac{1}{K}\left(-F_e +K_p\Delta y \right)+u^\star$$
where $F_{\rm est}$ is an estimate of $F$, $K_P=-10$. Set $K=0.1$, $T=2$, $y(0)=0.15$. The sampling time is $T_e=0.01s$. Fig. \ref{TirBf} displays results  with large uncertainities on parameters: $K_{system}=0.7 K_{model}$, $T_{system}=1.3 T_{model}$. The tracking remains excellent. The control variable on the other hand is now quite far from its nominal value.

\begin{figure*}[!ht]
\centering
\subfigure[\footnotesize $u^\star$ (- -) and $u$ (--)]
{\epsfig{figure=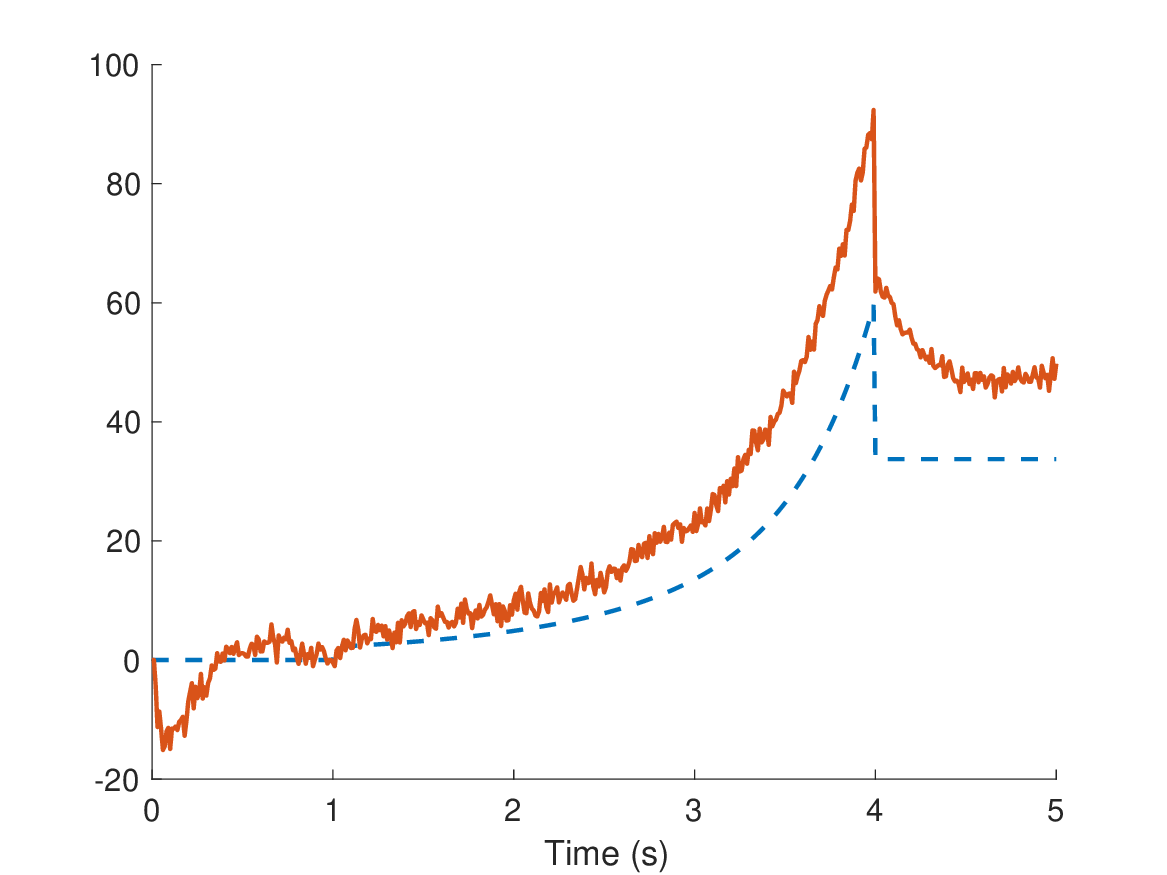,width=0.4\textwidth}}
\subfigure[\footnotesize $y^\star$ (- -) and $y$ (--)]
{\epsfig{figure=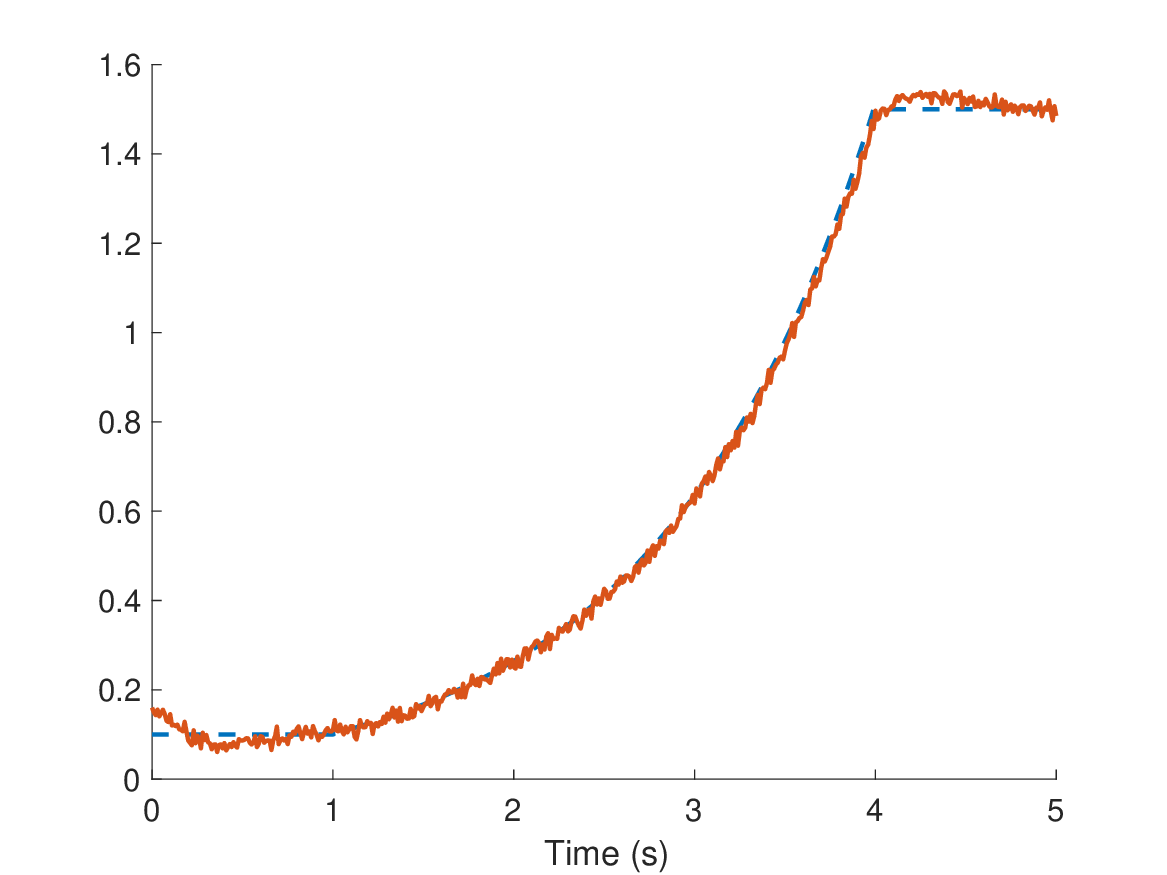,width=0.4\textwidth}}
\caption{Uncertain nonlinear system}\label{TirBf}
\end{figure*}

\section{Conclusion}\label{conclu}
Our approach to optimal control for LTI systems show results which were beyond the reach for LQRs, especially final conditions assignment and time horizon optimization. Don't they suggest that the infinite time horizon becomes pointless? Of course the proof which were for simplicity's sake restricted to monovariable systems should be extended to several control and output variables. The computer implementation ought also to be much more thoroughly investigated. 

The exploratory Section \ref{nl} on nonlinear optimal control should be developed in order to be more convincing.   

This text opens up new avenues for exploring the links between optimal control, robustness and various learning techniques when related to predictive control, key topics in today's literature (see, e.g., \cite{h,al,reinf,gros,scherer}).

\end{document}